\documentclass[11pt]{article}
\usepackage{graphics}
\usepackage{graphicx}
\usepackage{epsfig}
\usepackage{amsmath}
\DeclareMathOperator{\sign}{sign}
\usepackage{amsfonts}
\usepackage{amssymb}
\usepackage{xcolor}
\usepackage{enumerate}
\usepackage{subfigure}
\usepackage[export]{adjustbox}
\makeatletter
\@addtoreset{equation}{section}
\makeatother
\usepackage{pst-plot}
\psset{xunit=15mm}

\numberwithin{equation}{section}
\usepackage{tikz}
\makeatletter
\@namedef{subjclassname@2020}{\textup{2020} Mathematics Subject Classification}
\makeatother

\usepackage{a4wide,amsmath,amssymb,latexsym,amsthm}

\usepackage{enumitem}

\usepackage[colorlinks=true,urlcolor=blue,
citecolor=red,linkcolor=blue,linktocpage,pdfpagelabels,
bookmarksnumbered,bookmarksopen]{hyperref}
\usepackage[english]{babel}

\usepackage{setspace}

\newcommand{\1}{\mathbbm{1}}

\usepackage{authblk}    

\numberwithin{equation}{section}
\usepackage{mathrsfs,mathtools,epic,bm}
\usepackage{hyperref}
\hypersetup{colorlinks=true,linkcolor=blue,filecolor=mangeta,urlcolor= cyan}

\newtheorem{theorem}{Theorem}[section]
\newtheorem{proposition}[theorem]{Proposition}
\newtheorem{lemma}[theorem]{Lemma}
\newtheorem{corollary}[theorem]{Corollary}
\theoremstyle{definition}

\newtheorem{remark}[theorem]{Remark}

\def \dis {\displaystyle}

\def \R {\mathbb{R}}

\def \H {\mathcal{H}}

\def \A {\mathcal{A}}
\def \F {\mathbf{F}}
\def \D {\mathbf{D}}

\def \L {\mathcal{L}}
\def \BUC{\mathrm{BUC}}
\def \V {\mathcal{V}}
\def \U {\mathcal{U}}

\numberwithin{equation}{section}

\makeatletter
\renewcommand{\@seccntformat}[1]{\csname the#1\endcsname.\quad}
\makeatother

\def \dis {\displaystyle}
\title{Complete characterization of the sign of the wave speed in the symmetric Lotka-Volterra system under strong competition}

\author[]{Cyrille Kenne\thanks{Corresponding author}}

\affil[]{\footnotesize Department of Mathematics, University of British Columbia, Vancouver, BC V6T 1Z2, Canada}
\date{\today}
	
\begin{document}
\maketitle
\footnotetext{Email address: \href{mailto:kenne@math.ubc.ca}{kenne@math.ubc.ca}}

\begin{abstract}
This paper provides a complete characterization of the sign of the propagation speed in the symmetric two-species Lotka-Volterra competition-diffusion model under strong competition. The system admits a unique bistable travelling front and the sign of its speed determines which of the two species invades the other.
We prove that for every strong competition intensity greater than $1$ and every diffusion ratio $d\neq1$, the travelling front propagates so as to expand the territory occupied by the faster-diffusing species. The front has a zero speed exactly when $d=1$.  We also establish the smooth dependence of the wave speed and the travelling front on the model parameters.
The main step in the sign characterization is to prove that a monotone standing front cannot exist when the diffusion rates are different. Combined with continuity of the wave speed with respect to the parameters, the species-exchange symmetry of the system, and an explicit travelling front at a particular parameter value, we obtain the sign of the propagation speed throughout the entire parameter region. This establishes the ``Unity is not strength'' theorem, which was previously known only in restricted parameter regimes. 
\end{abstract}


\section{Introduction}
We study the symmetric competition-diffusion system 
\begin{equation}\label{model}
\left\{ 
\begin{array}{llllll}
	u_t&=&u_{xx}+u(1-u-kv) &\text{in}&(0,\infty)\times \R, \\
	v_t&=&dv_{xx}+v(1-v-ku)& \text{in}& (0,\infty)\times \R.
\end{array} 
\right.
\end{equation}
 in the strong competition regime $k>1$. The two species have the same growth, the same carrying-capacity and the same competition parameters. They differ only through the relative diffusion coefficient $d>0$.  These assumptions imply that the reaction term does not provide an intrinsic advantage to either species, allowing us to focus on the effects of the dispersal. Under this strong competition, the system \eqref{model} has two locally stable exclusion equilibria  $(1,0)$ and $(0,1)$. 
 The system \eqref{model} as well as its generalizations have been extensively studied in the literature. We refer for instance to \cite{carrere2018, gardner1982, girardin2019, guo2013, iannelli2015, kan1995} and the references therein.  
 
 The effect of movement on competitive success depends strongly on the spatial environment. In heterogeneous habitats and under blind competition (corresponding to $k=1$ here), Dockery et al. showed in \cite{dockery1998} that the slower diffuser excludes the faster one. This slow-dispersal advantage is usually interpreted as resulting from the greater retention of individuals in favourable regions. Girardin \cite{girardin2019} named this conclusion the ``Unity is strength'' principle. This means that one population dominates if and only if compared to the other, its individuals remain ``united''. 
 
 In homogeneous environments with strong competition, numerical and partial analytical results have consistently pointed in the opposite direction. There, the more rapidly diffusing population advances through a segregated contact zone \cite{alzahrani2010, alzahrani2012, girardin2019, girardin2018, ninomiya1995}. This means that it is better to randomly explore the hostile territory than to ``stay united'' even at the cost of a high mortality risk  for some individuals \cite{girardin2019}. 
 
 To understand this invasion process from a mathematical point of view, one can study solutions describing the propagation of the interface separating the two competing populations. In homogeneous environments, these propagation phenomena are naturally modelled by travelling waves (see e.g.,  \cite{volpert1994} for more details), namely solutions moving with constant speed on the infinite real line while maintaining their shape. The existence and uniqueness of travelling waves for partial differential equations has been extensively studied in the past (see e.g., \cite{gardner1982, tang1980,volpert1994}).  As pointed out in \cite{girardin2019}, the existence of a travelling wave provides only part of the answer to the ecological invasion problem. Indeed, once existence and uniqueness (up to translation) are established, the key quantity becomes the wave speed. Its sign determines which dispersal strategy successfully invades the other. However, determining this sign is generally much more delicate (see e.g., \cite{wang2022}). Consequently, only partial results are currently known.
 
 A monotone travelling front solution to \eqref{model}  with speed $c$ is a solution of the form,
 \begin{equation*}
 	(u,v)(t,x)=(U,V)(\xi), \qquad \xi=x-ct,
 \end{equation*}
such that
 \begin{equation}\label{eq1}
 	\begin{array}{llllll}
 	U''+cU'+U(1-U-kV)&=&0,\\
 	dV''+cV'+V(1-V-kU)&=&0, 
 	\end{array} 
 \end{equation}
 with 
 \begin{equation}\label{eq2}
 	(U,V)(-\infty)=(1,0), \qquad (U,V)(+\infty)=(0,1) 
 \end{equation}
 and 
  \begin{equation}\label{eq2a}
 U'<0, \qquad  V'>0 \; \text{ in } \R.
 \end{equation}
 Monotone here, means that \eqref{eq2a} is satisfied.  In our setting, $c<0$ means that $v$ chases $u$. Biologically, a negative speed indicates that $v$ dominates $u$ and $u$ becomes locally extinct, for initial conditions close to the travelling wave in a suitable function space (see e.g., \cite{gardner1982}). We note that this is a local invasion notion and does not determine the outcome for arbitrary initial data. When $d>1$, the statement ``Unity is not strength" holds true if and only if $c:=c_{k,d}<0$. A more precise question is: 
 
 \textit{If $k>1$ and $d>0$, is the bistable wave speed always of the sign corresponding to invasion by the faster diffuser? }
 
 The aim of this paper is to determine the sign of $c_{k,d}$. This global sign characterization has not yet been proved in general.  The main result of this paper is stated as follows.

 \begin{theorem}\label{maintheo}
 For every $k>1$ and $d>0$, 
 \begin{equation}\label{eq3}
\sign c_{k,d}=\sign(1-d).
 \end{equation}
 Equivalently,
 \begin{equation}\label{eq4}
 	c_{k,d}\left\{ 
 	\begin{array}{llllll}
 	>0, &  0<d<1,\\
 	=0, &  d=1,\\
 	<0, & d>1.
 	\end{array} 
 	\right.
 \end{equation}
 In particular, $c_{k,d}=0$ if and only if $d=1$.
 \end{theorem}

 \begin{figure}[h!]
 	\centering
 	\begin{tikzpicture}[x=2.0cm,y=1.2cm]
 		\fill[blue!10] (0.15,0) rectangle (1.2,3.6);
 		\fill[red!10] (1.2,0) rectangle (3.8,3.6);
 		\draw[->,thick] (0.15,0) -- (4.0,0) node[right] {$d$};
 		\draw[->,thick] (0.15,0) -- (0.15,3.85) node[above] {$k$};
 		\draw[very thick,black] (1.2,0) -- (1.2,3.6);
 		\draw (1.2,0.08) -- (1.2,-0.08) node[below=3pt] {$1$};
 		\draw (0.23,1) -- (0.07,1) node[left=3pt] {$1$};
 		\node[blue!60!black,align=center] at (0.53,2.0) {$c_{k,d}>0$\\fast $u$\\invades};
 		\node[red!65!black,align=center] at (2.45,2.0) {$c_{k,d}<0$\\fast $v$ invades};
 		\node[rotate=90,fill=white,inner sep=2pt] at (1.2,2.0) {$c_{k,1}=0$};
 		\node[anchor=west] at (0.22,3.5) {$k>1$};
 	\end{tikzpicture}
 	\caption{Complete sign diagram.  The line $d=1$ corresponds to the only standing-front  in the symmetric homogeneous model. }
 	\label{sign}
 \end{figure}
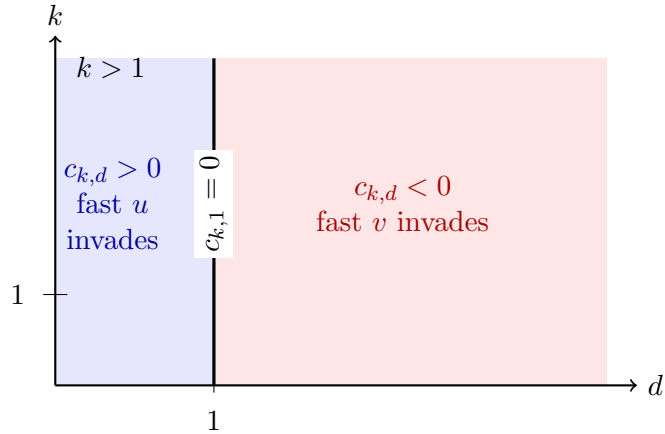

 In the sequel we will sometime use the following notations
 \begin{equation}\label{react}
F(U,V)=U(1-U-kV) \quad \text{ and } \quad G(U,V)=V(1-V-kU).
 \end{equation}
 
 \section{Related works and our contribution} 
 The classical existence theory for bistable travelling waves in competition-diffusion systems goes back to Gardner \cite{gardner1982}. Kan-On \cite{kan1995} established existence and uniqueness of travelling wave solutions for certain competition-diffusion equations. In a more general asymmetric Lotka-Volterra system, they established the monotonicity  of the propagation speed with respect to several reaction parameters, while keeping the diffusion ratio $d>0$ fixed. However, these results do not determine the dependence of the speed on the diffusion ratio $d$. 
 
 Determining the sign of the propagation speed in the Lotka-Volterra competition-diffusion system has therefore been the subject of substantial studies. Rodrigo and Mimura \cite{rodrigo2001} (see also \cite{rodrigo2000}) constructed nine exact families of travelling waves for particular choices of the parameters. Guo and Lin obtained in \cite{guo2013} explicit sufficient conditions determining the sign of the wave speed. Ma, Huang and Ou \cite{ma2019} later established additional explicit parameter regions in which the propagation direction can be determined .
 
 Other contributions treated asymptotic parameter regimes. Alzahrani, Davidson and Dodds proved in \cite{alzahrani2010} that the faster diffuser prevails when the diffusion ratio is sufficiently large. Girardin and Nadin \cite{girardin2015} obtained the same conclusion in the limit of very strong competition ($k\to \infty$). Risler \cite{risler2017} studied a singular regime near  $(d,k)=(1,1)$ and showed under an additional separation of scales,  that increasing mobility is advantageous. Girardin later synthesized these results and numerical evidences in \cite{girardin2019}. All the rigorously studied cases were consistent with the conjecture
\begin{equation*}
 c_{k,d}<0 \qquad\text{for every } d>1 \text{ and } k>1,
\end{equation*}
 but they did not provide a proof throughout the whole parameter region \cite{girardin2019}.
 
 Recent contributions enlarged the regions in which this result is known to hold. Chang, Chen and Wang introduced a minimax characterization of the zero-speed threshold and obtained new explicit
 sufficient conditions \cite{chang2023}. Morita, Nakamura and Ogiwara in \cite{morita2023} established further parameter regions in which the faster diffuser prevails. Xiao  \cite{xiao2025} proved the existence of regimes of positive and negative speed separated by a threshold in one of the competition coefficients. However, the corresponding thresholds were not estimated explicitly.
 
 More recently, Nakamura and Ogiwara developed a comparison method based on the construction of time-independent supersolutions \cite{nakamura2026}. In the symmetric system, they proved that $c_{k,d}<0$ on substantially larger parameter regions whose union is unbounded with respect to both $d$ and $k$ \cite[Theorem 4.1 and Figure 1]{nakamura2026}. Nevertheless, their conditions are sufficient and they explicitly left the global sign problem open. Chen and Wang \cite{chen2026} considered the regime $k\to1^+$. For every fixed $d>1$, they proved that the faster diffuser prevails when $k>1$ is sufficiently close to $1$~\cite[Theorem 1.2]{chen2026}. We mention that their convention for the speed is opposite to ours: if their speed is denoted by $s$, then $s=-c_{k,d}$. Thus, their conclusion $s>0$ corresponds to $c_{k,d}<0$ in our case.  A complete synthesis of regions in which the conjecture has been established is described in  \cite{nakamura2026}. We also mention the work of Alfaro and Xiao \cite{alfaro2023} for the critical case $k=1$.  Consequently, the faster-diffuser theorem had been verified at isolated parameter values, in several explicit interior regions, in the large-diffusion and large-competition regimes, and near the strong-competition boundary. However, a characterization valid throughout the entire parameter region $(d,k)\in(0,\infty)\times(1,\infty)$ is  still unavailable to our knowledge. 
 
The main purpose of this paper is to determine the propagation direction throughout the whole parameter region $(d,k)\in(0,\infty)\times(1,\infty)$. More precisely, we prove that $\sign c_{k,d}=\sign(1-d)$. In ecological terms, when two competing species have identical reaction parameters and differ only in their diffusion rates, the faster diffuser always advances through the contact zone. This establishes the ``Unity is not strength'' result for the symmetric homogeneous Lotka-Volterra system. The main difficulty is to identify the parameter values for which the propagation speed can vanish. One of the key steps which is also of independent interest, is the smooth dependence of the propagation speed and of the wave profile on both the competition intensity and the diffusion ratio. Using the later result, a change  of the sign of the wave speed could occur only through a parameter value at which $c_{k,d}=0$. Our second step is therefore a nonexistence result for monotone standing fronts when $d\neq1$ (see Theorem \ref{zerospeed}). This avoids estimating the speed directly and replaces the parameter-specific inequalities used in the previous literature with a single argument that applies throughout the parameter space.

The rest of the paper is organized as follows.  In Section \ref{smooth}, we recall some well known results on the existence of travelling waves for \eqref{model} and we prove the smooth dependence of the wave speed and the wave profile on both the competition intensity and the diffusion ratio. Section \ref{main1} establishes that  the existence of a standing monotone front is impossible unless $d = 1$. Finally, we prove our main result in Section \ref{mainresult}.
 
 \section{Smooth dependence on the parameters $k$ and $d$}\label{smooth}
 The aim of this section is to prove that the propagation speed $c_{k,d}$ and the wave profile $(U_{k,d}, V_{k,d})$ depend smoothly on $k$ and  $d$. This is the first key step towards Theorem \ref{maintheo}  (only the continuity is actually needed). We mention that Kan-On \cite{kan1995} obtained the $C^1$ dependence of the speed on the reactions parameters at fixed diffusion ratio, together with the monotonicity of the speed with respect to each of these reactions parameters. Their results do not include the diffusion ratio.
 
 For further needs we recall the following now well-known result. Its proof is contained in \cite[Theorem 2.1]{kan1995} (see also \cite[Theorem 2.3]{girardin2015}). 
 The proof of the bounds \eqref{bound} follows from the maximum principle (see e.g., \cite[Lemma 3.1]{alfaro2023}).

\begin{theorem}\label{theo1}
	For every $k>1$ and $d>0$, there exists a unique $c_{k,d}$ for which \eqref{eq1}-\eqref{eq2a} has a solution. Its profile $(U,V)$ is unique up to a translation, belongs to $C^\infty(\R)^2$ and satisfies
	\begin{equation}\label{bound}
		0<U(\xi), V(\xi)<1\quad \text{ for every } \xi\in \R.
	\end{equation}
	Moreover, $-2\sqrt{d}<c_{k,d}<2$. Furthermore, if $c\in\R$ and if $(\widetilde{U},\widetilde{V})\in C^2(\R)^2$ is an arbitrary positive solution of \eqref{eq1} with speed $c$ satisfying
	\begin{equation}\label{limits}
		(\widetilde{U},\widetilde{V})(-\infty)=(1,0),\qquad (\widetilde{U},\widetilde{V})(+\infty)=(0,1)
	\end{equation}
	then $c=c_{k,d}$ holds and there exists $\tau\in\R$ such that
	\begin{equation}\label{transl}
		(\widetilde{U},\widetilde{V})(\xi)=(U,V)(\xi+\tau)\qquad\text{for every }\xi\in\R .
	\end{equation}
\end{theorem}


Let $\Phi:=(U_{k, d}, V_{k,d})^T$ be a classical front  satisfying \eqref{eq1}-\eqref{eq2a} with speed $c:=c_{k,d}$.  We write \eqref{eq1}-\eqref{eq2} as 
\begin{equation}\label{comp}
	\D_d \Phi''+c \Phi^{\prime}+\F_k(\Phi)=0, \qquad \Phi(-\infty) = E_- := \begin{pmatrix}1\\0\end{pmatrix}, \quad \Phi(+\infty) = E_+ := \begin{pmatrix}0\\1\end{pmatrix}, 
\end{equation}
where
\begin{equation}\label{comp1}
\D_d = \begin{pmatrix}1&0\\0&d\end{pmatrix}
\qquad\text{and}\qquad
\F_k(\Phi) = \begin{pmatrix} U(1-U-kV) \\ V(1-V-kU)\end{pmatrix}.
\end{equation}

\subsubsection*{Notations}Throughout the paper, we use standard Sobolev spaces notations. We write $L^2(\R)^2:=L^2(\R;\R^2)$ and $H^2(\R)^2:=H^2(\R;\R^2)$. We denote by $C_b(\R)^2:=C_b(\R;\R^2)$ the space of bounded continuous functions from $\R$ to $\R^2$, and by $\BUC:=\BUC(\R;\R^2)$ the space of bounded uniformly continuous functions from $\R$ to $\R^2$. $C_b^1(\R)^2:=C_b^1(\R; \R^2)=\left\{h \in C^1(\R; \R^2): h\right.$ and $h^{\prime}$ are bounded $\}$. For $z\in\R^2$, $\|z\|$ denotes the Euclidean norm. For $f,g\in L^2(\R)^2$, we denote their $L^2$ inner product by $\langle f,g\rangle_{L^2}$. We denote by $\|\cdot\|_{L^\infty}$ the norm in $L^\infty(\R)^2$ and by $\|\cdot\|_{H^2}$ the norm in $H^2(\R)^2$.  For a linear operator $T$, we write $\ker T$ and $\operatorname{Ran}T$ for its kernel and range, respectively, and $\operatorname{span}\{h\}$ for the linear space generated by $h$. Finally, $D\F(\Phi)$ denotes the Jacobian matrix of $\F$ evaluated at $\Phi$.

We have  the following result.
\begin{theorem}\label{theosmooth} Let $k>1$ and $d>0$. Let  $c_{k,d}$ be the speed of the unique monotone front $\Phi_{k, d}=(U_{k,d}, V_{k,d})^T$ of \eqref{eq1}-\eqref{eq2} associated with $(k,d)$ and normalized by $U_{k,d}(0)=\frac{1}{2}$. Then, the speed map
	\begin{equation*}
		(k,d)\longmapsto c_{k,d}
	\end{equation*}
	is $C^\infty$ on $(1,\infty)\times(0,\infty)$. 
	Moreover, for every
	$(k_0,d_0)\in(1,\infty)\times(0,\infty)$, there exists a neighbourhood
	$\V_0$ of $(k_0,d_0)$ such that the map
	\begin{equation*}
		(k,d)\longmapsto
		\Phi_{k,d}-\Phi_{k_0,d_0}\in H^2(\R)^2
	\end{equation*}
	is of class $C^\infty$ on $\V_0$. 
	
\end{theorem}

Before proving the Theorem \ref{theosmooth}, we state and prove the following results.

\begin{lemma}\label{decay} 	 Let $k_0>1$ and $d_0>0$ and let $(\Phi_0, c_0)$ be a monotone front associated with $k_0$ and $d_0$. Then there exists $C,\eta>0$ such that
	\begin{equation}\label{cy0}
		\|\Phi_0(\xi)-E_-\| + \|\Phi_0'(\xi)\|  \leq C e^{\eta\xi}, \qquad \xi\leq 0
	\end{equation}
	and 
	\begin{equation}\label{cy1}
		\|\Phi_0(\xi)-E_+\| + \|\Phi_0'(\xi)\|  \leq C e^{-\eta\xi}, \qquad \xi\geq 0.
	\end{equation}
	Moreover, for every integer $m\geq 1$, there exists $C_m>0$ such that
	\begin{equation}\label{generaldecay}
		\|\Phi_0^{(m)}(\xi)\|\leq C_me^{-\eta|\xi|}, \qquad \text{ for all } \xi\in \R.
	\end{equation}
	In particular,
	\begin{equation}\label{cy2}
		\Phi'_0\in H^2(\R)^2, \qquad \Phi_0'' \in L^2(\R)^2, \qquad \F_{k}(\Phi_0)\in L^2(\R)^2,
	\end{equation}
	for every fixed $k>1$. 
\end{lemma}

\begin{proof}
	The proof of \eqref{cy0}-\eqref{cy1} follows essentially similar arguments as in \cite[Lemmas 3.2 and 3.6]{gardner1982}. We provide here the sketch of the proof. We let $Y=(\Phi_0,\Phi_0')\in \R^4$. Then, we can rewrite \eqref{comp} as the following first-order system 
	\begin{equation}\label{firstordersyst}
		Y'=\mathcal{B}(Y), \qquad  \mathcal{B}\begin{pmatrix}z\\p\end{pmatrix}= \begin{pmatrix} p \\ -\D_{d_0}^{-1}\left[c_0p+\F_{k_0}(z)\right]\end{pmatrix}.
	\end{equation}
	The characteristic polynomial of the linearization at $(E_-, 0)$ is
	\begin{equation}\label{charact1}
		(r^2+c_0r-1)(d_0r^2+c_0r+1-k_0)=0,
	\end{equation}
	and at $(E_+,0)$ 
	\begin{equation}\label{charact2}
		(r^2+c_0r+1-k_0)(d_0r^2+c_0r-1)=0.
	\end{equation}
	For $k_0>1$ and $d_0>0$, every quadratic term in \eqref{charact1}-\eqref{charact2} has two real nonzero roots of opposite sign. So the two endpoint equilibria of \eqref{firstordersyst} are hyperbolic, with two stable and two unstable eigenvalues. The exponential estimates in the stable and unstable manifold theorem give \eqref{cy0}-\eqref{cy1} (possibly after decreasing $\eta$).  We now prove \eqref{generaldecay}. For $m=1$, the estimate follows directly from \eqref{cy0}-\eqref{cy1}. Next, we write
	\begin{equation}\label{neweq0}
	\Phi_0''=-\D_{d_0}^{-1}\left[c_0\Phi_0'+\F_{k_0}(\Phi_0)\right].
	\end{equation} 
	Since $\F_{k_0}(E_{\pm})=0$ and $\F_{k_0}$ is $C^1$, we can deduce from \eqref{cy0}-\eqref{cy1} that 
	$|\F_{k_0}(\Phi_0(\xi))| \leq C e^{-\eta|\xi|}, \;\xi \in \mathbb{R}$. Consequently, $|\Phi_0''(\xi)| \leq C_2 e^{-\eta|\xi|}$ for all $\xi \in \mathbb{R}$ and for some $C_2>0$. Hence \eqref{generaldecay} holds for $m=2$. Differentiating \eqref{neweq0} and using the fact that $\Phi_0$ bounded implies $D \F_{k_0}(\Phi_0)$ is also bounded, we deduce \eqref{generaldecay} for $m=3$. 
	Note that $\F_{k_0}$ is quadratic, so its derivatives of order three and higher vanish and its second derivative is constant.  Therefore,  differentiating \eqref{neweq0} and using an induction argument, we can deduce \eqref{generaldecay} for $m\geq4$.  Finally, using \eqref{generaldecay}, we deduce the first two assertions in \eqref{cy2}. Let $k>1$, we have $\F_{k}\left(E_{ \pm}\right)=0$ and $\F_{k}$ is of class $C^1$. Thus,  the exponential estimates \eqref{cy0}-\eqref{cy1} imply that 
	$\F_{k}(\Phi_0) \in L^2(\mathbb{R})^2$. This completes the proof.
\end{proof}

\begin{proposition}\label{propposi}
	Let $k_0>1$ and $d_0>0$, and let $\Phi_0=(U_0,V_0)$ be a monotone front of \eqref{eq1}-\eqref{eq2} associated with $(k_0,d_0)$ and  normalized by $U_0(0)=\frac{1}{2}$. Then there exist $\varepsilon_0>0$ and $\delta_0>0$, depending only on $k_0$ and $d_0$, with the following property. Let $k>1$, $d>0$ and $c\in\R$ be such that $|k-k_0|\leq\delta_0$, and let $\Phi=(U,V)\in C^2(\R)^2$ be a solution of \eqref{eq1}-\eqref{eq2} associated with $(k,d)$ and  with speed $c$ such that
	\begin{equation}\label{close}
\|\Phi-\Phi_0\|_{L^\infty}\leq\varepsilon_0.
	\end{equation}
	Then, $U>0$ and $V>0$ on $\R$.
\end{proposition}

\begin{proof} We first note that the normalization $U_0(0)=\frac{1}{2}$ selects a unique translate of the monotone front (see e.g., \cite[Theorem 2.3]{girardin2015}). Therefore, $\Phi_0$ is uniquely determined by the parameters $(k_0, d_0)$. Let $\delta_0:=\frac{k_0-1}{2}$ and set $k_1:=\min\left\{\frac{k_0+1}{2},2\right\}$. Then $1<k_1\leq 2$ and for all $k$ such that $|k-k_0|\leq\delta_0$, we have $k\geq k_1$.	 Set $\eta:=\frac{k_1-1}{8}$, then $0<\eta\leq \frac{1}{8}$. 
	Since $\Phi_0$ satisfies the boundary conditions $\eqref{eq2}$, there exists $M>0$ such that
	\begin{equation}\label{mum0}
		U_0\geq 1-\eta \ \text{ on } (-\infty,-M],\qquad V_0\geq 1-\eta \ \text{ on } [M,+\infty).
	\end{equation}
	The functions $U_0$ and $V_0$ are continuous and positive on $\R$, so by compactness
	\begin{equation}\label{defm0}
		m_0:=\min_{\xi\in[-M,M]}\min\{U_0(\xi),V_0(\xi)\}>0 .
	\end{equation}
	We note that the number $M$ in \eqref{mum0} and the number $m_0$ in \eqref{defm0} depend on $k_0$ and on $\Phi_0$. Since $\Phi_0$ is uniquely determined by $(k_0, d_0)$, we have $\varepsilon_0=\varepsilon_0(k_0, d_0)$ and $\delta_0=\delta_0(k_0)$. 
	Finally, we take
	\begin{equation}\label{defeps0}
		\varepsilon_0:=\min\left\{\frac{m_0}{2}, \eta\right\}.
	\end{equation}
Now from \eqref{close}, we can deduce that
	\begin{equation}\label{compa0}
		|U(\xi)-U_0(\xi)|\leq\varepsilon_0\quad\text{and}\quad |V(\xi)-V_0(\xi)|\leq\varepsilon_0,\qquad \xi\in\R .
	\end{equation}
	By combining \eqref{defm0}-\eqref{compa0}, we obtain on $[-M,M]$ that
	\begin{equation}\label{ewen1}
		U\geq U_0-\varepsilon_0\geq m_0-\frac{m_0}{2}=\frac{m_0}{2}>0,\qquad
		V\geq V_0-\varepsilon_0\geq\frac{m_0}{2}>0.
	\end{equation}
	Moreover, combining \eqref{mum0}-\eqref{compa0}, we deduce that
	\begin{equation}\label{ewen2}
		U\geq U_0-\varepsilon_0\geq 1-\eta-\varepsilon_0\geq1-2\eta>0\qquad\text{ on }(-\infty,-M],
	\end{equation}
	and in the same way
	\begin{equation}\label{ewen2a}
		V\geq V_0-\varepsilon_0\geq 1-\eta-\varepsilon_0\geq 1-2\eta>0\qquad\text{ on }[M,+\infty).
	\end{equation}
	It remains to show that $U>0$ on $[M,+\infty)$ and that $V>0$ on $(-\infty,-M]$.
	We set
	\begin{equation}\label{defqs}
		q_U:=1-U-kV,\qquad q_V:=1-V-kU,
	\end{equation}
	which are continuous on $\R$. On $[M,+\infty)$ we have $U\geq U_0-\varepsilon_0\geq-\varepsilon_0\geq -\eta$ (because $U_0>0$)
and  from above, we have $V\geq 1-2\eta>0$. Since $k\geq k_1$, we obtain
	\begin{equation}\label{ewen3}
		q_U\leq 1+\eta-k_1(1-2\eta)= -(k_1-1)+\eta(1+2k_1)\leq -8\eta+5\eta=-3\eta<0\qquad\text{on }[M,+\infty).
	\end{equation}
	Similarly, on $(-\infty,-M]$ we have $V\geq V_0-\varepsilon_0\geq-\varepsilon_0\geq -\eta$, and
	$U\geq 1-2\eta>0$, so with the same computations as above, we obtain
	\begin{equation}\label{ewen4}
		q_V\leq 1+\eta-k_1(1-2\eta)\leq -3\eta<0\qquad\text{on }(-\infty,-M].
	\end{equation}
	
On the one hand, using \eqref{defqs}, we write the first equation in \eqref{eq1} as 
	\begin{equation}\label{eqU}
		U''+cU'+q_UU=0\qquad\text{on }[M,+\infty)
	\end{equation}
with the following conditions
	\begin{equation}\label{bcU}
		U(M)\geq\frac{m_0}{2}>0\qquad\text{and}\qquad U(+\infty)=0,
	\end{equation}
	where the first inequality comes from above and the second one from \eqref{eq2}. The coefficient $q_U$ is continuous and negative on $[M,+\infty)$ by \eqref{ewen3}. Therefore,  we apply the Lemma \ref{lempos} to obtain
	\begin{equation}\label{ewen5}
		U>0\qquad\text{on }[M,+\infty).
	\end{equation}
	On the other hand, from the second equation of $\eqref{eq1}$, we have
	\begin{equation}\label{eqV}
		dV''+cV'+q_VV=0\qquad\text{on }(-\infty,-M],
	\end{equation}
	together with
	\begin{equation}\label{bcV}
		V(-M)\geq\frac{m_0}{2}>0\qquad\text{and}\qquad V(-\infty)=0.
	\end{equation}
The coefficient $q_V$ is continuous and negative on $(-\infty,-M]$ by \eqref{ewen4}. Since $d>0$, we apply again the Lemma \ref{lempos} with $w(\xi)=V(-\xi)$ to obtain that
	\begin{equation}\label{ewen6}
		V>0\qquad\text{on }(-\infty,-M].
	\end{equation}
	Combining  \eqref{ewen1}, \eqref{ewen2}, \eqref{ewen2a}, \eqref{ewen5} and \eqref{ewen6} lead us to $U>0$ and $V>0$ on $\R$. This completes the proof.
\end{proof}

We can now prove the main result of this section. 
\begin{proof}[Proof of Theorem \ref{theosmooth}]	We fix $\left(k_0, d_0\right) \in(1, \infty) \times(0, \infty)$  and we let $c_0:=c_{k_0, d_0}$. We denote by  $\Phi_0=\left(U_0, V_0\right)^T:=\Phi_{k_0, d_0}$ the corresponding monotone front satisfying \eqref{comp}-\eqref{comp1} and normalized by $U_0(0)=\frac{1}{2}$.  We first observe that the set of fronts satisfying  \eqref{comp}-\eqref{comp1} is not a vector space. We use $\Phi=\Phi_0+w$, where $w$ will be required to vanish at both endpoints. We recall that in dimension $1$, $H^2(\mathbb{R})^2\hookrightarrow C^1_b(\mathbb{R})^2$; so if $w\in H^2(\mathbb{R})^2$ then $w(\xi), w'(\xi)\to 0$ as $|\xi|\to\infty$. Next,  set $w=(w_1,w_2)^T$ and introduce the map 
\begin{equation}\label{defh}
	\H(w,c,k,d) := \Big(\D_d(\Phi_0+w)'' + c(\Phi_0+w)' + \F_k(\Phi_0+w),\ w_1(0)\Big)
\end{equation}
from $H^2(\mathbb{R})^2\times\mathbb{R}\times\mathbb{R}^2 \longrightarrow L^2(\mathbb{R})^2\times\mathbb{R}$.	
We observe that the map  $\ell: w\mapsto \ell(w):=w_1(0)$ defined from $H^2(\R)^2 \to \R$ is linear and bounded because $|\ell(w)|=|w_1(0)| \leq\left\|w_1\right\|_{L^{\infty}} \leq C_S\left\|w_1\right\|_{H^1} \leq C_S\|w\|_{H^2}$. Hence, the second component of $\H$ is well defined.
Next, from the Lemma \ref{decay}, we can deduce with straightforward computations that the first component of $\H$ is well defined. Thus $\H$ is well defined. Moreover, $\H$ is polynomial in $(w,c,k)$ and affine in $d$, with bounded multilinear terms. Hence $\H$ is of class $C^\infty$. In addition,
\begin{equation}\label{cy3}
\H(0,c_0,k_0,d_0) = (0,0).
\end{equation}
Moreover, the derivative of $\H(w,c,k,d)$ with respect to $(w,c)$ at $(0,c_0,k_0,d_0)$ is given by:
\begin{equation}\label{deriv}
\A(h,s) := D_{(w,c)}\H(0,c_0,k_0,d_0) = \big(\L_0 h + s\Phi_0',h_1(0)\big),
\end{equation}
where
\begin{equation}\label{defl}
\L_0h:=\D_{d_0}h'' + c_0 h' + D\F_{k_0}(\Phi_0)h.
\end{equation} 
We show that the linear bounded operator  $\A: H^2(\R)^2\times\R \longrightarrow L^2(\R)^2\times\R$ is an isomorphism. 

Before proceeding, we adapt some facts established in \cite{kan1996} concerning the operator $\L_0$. It is worth mentioning that their results is shown to hold on $\BUC$, but here the operator $\L_0:H^2(\R)^2\to L^2(\R)^2$. After the reflection $\xi \mapsto-\xi$, which changes the wave speed from $c_0$ to  $-c_0$, our wave system is exactly a particular case of theirs, with $a=1$ and $b=c=k_0$.  We observe that a differentiation of the equation in \eqref{comp} gives that $\L_0\Phi_0'=0$. In addition, if $h\in H^2(\R)^2$, then $h,h'\in \BUC$ and $\L_0h=0$ implies $h'' = -\D_{d_0}^{-1}\big(c_0h' + DF_{k_0}(\Phi_0)h\big) \in \BUC$ (note that $d_0>0$). So, $h$ is a bounded classical solution of $\L_0 h=0$. Consequently, after the reflection $\xi \mapsto-\xi$,  \cite[Lemma 3.3]{kan1996} applies and gives us 
\begin{equation}\label{cy4}
\operatorname{ker} \L_0=\operatorname{span}\{\Phi_0'\}.
\end{equation}
We also claim that $\Phi_0'\notin \operatorname{Ran} \L_0$. Indeed, if we suppose that there exists $h\in H^2(\R)^2$ such that $\L_0h=\Phi_0^{\prime}$, then one uses the same arguments as above to deduce that $h$ is a bounded classical solution of $\L_0h=\Phi_0^{\prime}$. Next, we use the reflection $\widehat{\Phi}_0(\xi)=\Phi_0(-\xi), \quad \widehat{h}(\xi)=-h(-\xi)$ and we set $\widehat\L_0 h:=\D_{d_0} h''-c_0 h'+D \F_{k_0}(\widehat{\Phi}_0(\xi)) h$. Then, it follows that $\widehat\L_0 \widehat{h}=\widehat{\Phi}_0'$ and $\widehat\L_0 \widehat{\Phi}_0'=0$. Therefore, $\widehat\L_0^2 \widehat{h}=0$ and $\widehat\L_0 \widehat{h} \neq 0$. This contradicts the algebraic simplicity of the eigenvalue $0$ \cite[Page 348]{kan1996}.
Hence,
\begin{equation}\label{cy5}
	\Phi_0' \notin \operatorname{Ran} \L_0. 
\end{equation}
Furthermore, $\L_0$ is Fredholm of index $0$. Indeed, the first-order matrix coefficients at $\pm\infty$ have characteristic polynomials $(r^2+c_0r-1)(d_0r^2+c_0r+1-k_0)=0$ and $(r^2+c_0r+1-k_0)(d_0r^2+c_0r-1)=0$. Thanks to the fact that $d_0>0$ and $k_0>1$, we deduce that each of the quadratic terms has one positive and one negative root.
Hence these matrices at $\pm\infty$ have each a Morse index equal to $2$. The Fredholm theorem for asymptotically hyperbolic differential operators (see e.g., \cite[Theorem 3.2, Remark 3.3, and Theorem 3.3]{sandstede2002}) gives that the index of the operator $\L_0$ is equal to $2-2 = 0$.  Moreover, since $\operatorname{dim}\operatorname{ker} \L_0=1$, it follows that $\operatorname{codim} \operatorname{Ran} \L_0=1$. This combined with \eqref{cy4} and \eqref{cy5} gives
$$L^2(\mathbb{R})^2 = \operatorname{Ran}\L_0 \oplus \operatorname{span}\{\Phi_0'\}.$$
With standard arguments, we conclude that $\A$ is surjective. 	Now if $\A(h,s)=0$, then $\L_0h+s\Phi_0'=0$ and $h_1(0)=0$. If $s\neq 0$, then $\Phi_0' = -\L_0(h/s) \in \operatorname{Ran}\L_0$, contradicting \eqref{cy5}. So, $s=0$ and thus $\L_0h=0$. This implies that $h\in\operatorname{ker}\L_0$, which by \eqref{cy4} gives $h=\alpha\Phi_0'$, for some $\alpha\in\mathbb{R}$. $h_1(0)=0 \iff \alpha U'_0(0)=0$. But $U'_0(0)<0$ from \eqref{eq2a}. Therefore $\alpha=0$ and so $h=0$. Hence $\A$ is injective and therefore bijective. From the bounded inverse theorem, we obtain that $\A$ an isomorphism. The implicit function theorem gives the existence of neighbourhoods $\V\subset (1,\infty)\times (0,\infty)$  of $(k_0,d_0)$ and $\U\subset H^2(\R)^2\times \R$ of 
	$(0,c_0)$ and a unique $C^\infty$ map
$$
	(k,d)\longmapsto (w(k,d),c(k,d))\in H^2(\R)^2\times\R
$$
such that $w(k_0, d_0)=0$, $c(k_0, d_0)=c_0$ and $\H(w(k, d), c(k, d), k, d)=(0,0)$ for every $(k,d)\in \V$. Therefore, by setting

\begin{equation}\label{ken000}
	\widetilde{\Phi}_{k,d}:=\Phi_0+w(k,d)=(\widetilde{U}_{k, d}, \widetilde{V}_{k, d})^T,
\end{equation} we can use bootstrap arguments to deduce that  $\widetilde{\Phi}_{k,d}$ is a classical solution of \eqref{comp}. In addition, $w_1(0)=0$ implies that $\widetilde{U}_{k, d}(0)=U_0(0)=\frac{1}{2}$. Now,  it remains to prove that  $\widetilde{\Phi}_{k,d}$  is positive. Note that this does not follow from the construction. Without loss of generality, we may assume that  $\V:=\left\{(k, d) \in(1, \infty) \times(0, \infty):\left|k-k_0\right|+\left|d-d_0\right|<r\right\}$, for some $r>0$.  Let $\varepsilon_0>0$ and $\delta_0>0$ be the constants given by the Proposition \ref{propposi}. Thanks to the Sobolev embedding $H^2(\R)^2\hookrightarrow C_b(\R)^2$, we can deduce the existence of a constant $C>0$ so that 
\begin{equation}\label{sobineq}
	\|w\|_{L^{\infty}}\leq C\|w\|_{H^2},\qquad w\in H^2(\R)^2 .
\end{equation}
 Since the map $(k,d)\mapsto w(k,d)$ is continuous from $\V$ into $H^2(\R)^2$ with $w(k_0,d_0)=0$, we obtain that there exists $\delta_1\in (0,r]$ such that
	$|k-k_0|+|d-d_0|<\delta_1$ implies $\|w(k,d)\|_{H^2}\leq\frac{\varepsilon_0}{C}$.
Since $\widetilde{\Phi}_{k,d}-\Phi_0=w(k,d)$, we deduce from \eqref{sobineq} that there exists $\delta_1\in (0,r]$ so that 	$|k-k_0|+|d-d_0|<\delta_1$ implies 
\begin{equation}\label{ken0}
	\|\widetilde{\Phi}_{k,d}-\Phi_0\|_{L^\infty}\leq\varepsilon_0.
\end{equation}
We set $\V_0:=\left\{(k,d)\in(1,\infty)\times(0,\infty)\ :\ |k-k_0|+|d-d_0|<\min\{\delta_0,\delta_1\}\right\}$, 
so that $\V_0\subset\V$ is an open neighbourhood of $(k_0,d_0)$. Let $(k,d)\in\V_0$. Then $|k-k_0|\leq\delta_0$, and \eqref{ken0} holds. Proposition \ref{propposi} then applies and gives
\begin{equation*}
	\widetilde{U}_{k,d}>0\quad\text{and}\quad \widetilde{V}_{k,d}>0\ \text{ on }\R,\qquad (k,d)\in\V_0 .
\end{equation*}
Hence, $\widetilde{\Phi}_{k,d}$ is a positive solution of \eqref{eq1}-\eqref{eq2} with speed $c(k,d)$. Theorem \ref{theo1} implies that
\begin{equation}\label{cident}
	c(k,d)=c_{k,d},\qquad (k,d)\in\V_0 ,
\end{equation}
and that $\widetilde{\Phi}_{k,d}=\Phi_{k, d}(\cdot+\tau)$ for some $\tau\in \R$. 
Thanks to the fact that $\widetilde{U}_{k, d}(0)=\frac{1}{2}$, we obtain $U_{k, d}(\tau)=\frac{1}{2}$. Since $U_{k, d}$ is strictly decreasing by \eqref{eq2a}, the equation $U_{k,d}(\xi)=\frac{1}{2}$ has at most one solution, and $\xi=0$ is one by the normalisation of $\Phi_{k,d}$. Hence $\tau=0$ and 
\begin{equation}\label{profilechoice}
	\Phi_{k,d}=\widetilde{\Phi}_{k,d},\qquad (k,d)\in\V_0 .
\end{equation}

The map $(k,d)\mapsto c(k,d)$ being of class $C^\infty$ on $\V$, \eqref{cident} implies that $(k,d)\mapsto c_{k,d}$ is of class $C^\infty$ on $\V_0$. Moreover, by \eqref{profilechoice} we have $\Phi_{k,d}=\Phi_0+w(k,d)$ for every $(k,d)\in \V_0$ and the map $(k, d) \mapsto w(k, d) \in H^2(\mathbb{R})^2$ is of class $C^{\infty}$ on $\mathcal{V} \supset \mathcal{V}_0$ from above. Hence, $	\Phi_{k,d}-\Phi_{k_0,d_0}\in H^2(\R)^2$ is of class $C^\infty$ on $\V_0$.
Since $(k_0,d_0)$ was arbitrary in $(1,\infty)\times(0,\infty)$, we conclude that $(k,d)\mapsto c_{k,d}$ is of class $C^\infty$ on $(1,\infty)\times(0,\infty)$. 
This completes the proof.
\end{proof}

The next result states a classical well-known result.
\begin{lemma}\label{lem1}
Let $I\subset\R^n$ be a nonempty and connected set. Let $f: I\to \R$ be a continuous function so that $f(x)\neq 0$ for every $x\in I$. Then either $f>0$ on $I$ or $f<0$ on $I$.
\end{lemma}

\section{Necessary condition for the existence of a standing front} \label{main1}
In this section, we prove that for any monotone travelling front of \eqref{model} with zero wave speed $c=0$, one has necessarily $d=1$. 
 The main result of this section is the following.
\begin{theorem}\label{zerospeed}
Let $k>1$ and $d>0$. Let $(U,V)$ be a solution of \eqref{eq1}-\eqref{eq2a} with speed  $c_{k,d}$.  If $c_{k,d}=0$, then
\begin{equation} \label{imp1}
	(1-d)\int_{\R}U''(\xi)V'(\xi)d \xi =0
\end{equation}
and 
\begin{equation}\label{imp2}
	\int_{\R}U''(\xi)V'(\xi)d \xi=\frac{1}{2d} \int_{\mathbb{R}} (1-U(\xi)-V(\xi))^2(-U'(\xi)) d \xi \geq \frac{(k-1)^2}{24dk^2}>0
\end{equation}
hold. Consequently, $d=1$. 
\end{theorem}

\begin{proof} The proof relies on scalar energy identities. We mention that the identities \eqref{e01} and \eqref{e02} are standard and closely related to the one derived in \cite{guo2013}. The main new point here is their combination with \eqref{e04} to obtain \eqref{imp1} and the strict positivity estimate \eqref{imp2}. 
	 If we take $c=0$ in \eqref{eq1}, we obtain 
\begin{equation}\label{eq1new}
	\begin{array}{llllll}
		U''+U(1-U-kV)&=&0,\\
		dV''+V(1-V-kU)&=&0.
	\end{array} 
\end{equation}
Multiplying the first equation in \eqref{eq1new} by $U'$ and integrating on $\R$, we obtain 
\begin{equation}\label{e00}
	\int_{\R}U''U'd\xi+\int_{\R}UU'(1-U)d\xi -k\int_{\R}UVU'd\xi=0.
\end{equation}
Using the boundary limits in Lemma \ref{decay}, an integration by parts shows that the first term of \eqref{e00} vanish and the second term 
\begin{equation}\label{e000}
 \int_{\R}UU'(1-U)d\xi=-\frac{1}{6}.
\end{equation}
Therefore,
\begin{equation}\label{e01}
\int_{\R}UVU'd\xi=-\frac{1}{6k}.
\end{equation}
Multiplying the second equation in \eqref{eq1new} by $V'$ and employing similar arguments as above, we can deduce that
\begin{equation}\label{e02}
	\int_{\R}UVV'd\xi=\frac{1}{6k}.
\end{equation}

Combining \eqref{e01} and \eqref{e02} gives
\begin{equation}\label{e03}
\int_{\R}UVU'd\xi+	\int_{\R}UVV'd\xi=0.
\end{equation}
Next, we claim that  
\begin{equation}\label{e04}
 \int_{\R}[F(U,V)V'+G(U,V)U']d\xi =0.
\end{equation}
Indeed, let $R>0$. Using an integration by part, we obtain
\begin{equation*}
	\begin{array}{lllll}
\dis 	\int_{-R}^R[F(U,V)V'+G(U,V)U']d\xi &=&  [UV]^R_{-R}-[U^2V]_{-R}^R-[V^2U]_{-R}^R\\
&&+\dis (2-k)\int_{-R}^{R}(UVU'+UVV')d\xi.
	\end{array}
\end{equation*}
Letting $R\to +\infty$ in the previous identity while using Lemma \ref{decay} and \eqref{e03} proves the claim.

Now, we multiply the first and the second equation in \eqref{eq1new} by $V'$ and $U'$, respectively, combine the two equations and we integrate by parts over $[-R,R]$ to obtain
\begin{equation}\label{e05}
\int_{-R}^RU''V'd\xi +d\int_{-R}^RV''U'd\xi + \int_{-R}^R[F(U,V)V'+G(U,V)U']d\xi =0.
\end{equation}
An integration by part of the second term in \eqref{e05} gives
\begin{equation}\label{e06}
d\int_{-R}^RV''U'd\xi =d[V'U']_{-R}^R-d\int_{-R}^{R}V'U''d\xi.
\end{equation}
Combining \eqref{e05} and \eqref{e06} and taking the limit as $R\to +\infty$, while using \eqref{e04} gives
$$	(1-d)\int_{\R}U''V'd \xi =0.$$ Hence \eqref{imp1} holds. 
From \eqref{e06}, letting $R\to \infty$ and combining the result with the second equation of \eqref{eq1new} implies that
\begin{equation}\label{siha0}
\dis \int_{\R}U''V'd \xi =-\int_{\R}U'V''d \xi=\frac{1}{d}\int_{\R}G(U,V)U'd\xi.
\end{equation}
Let us set $W=1-U-V$. Next, we write $2G(U,V)+W^2=(1-U)^2-V^2-2(k-1)UV$. Then multiplying this identity by $U'$ and integrating by part over $\R$ leads us to 
 \begin{equation}\label{e07}
2\int_{\R}G(U,V)U'd\xi +\int_{\R}W^2U'=\int_{\R}(1-U)^2U'd\xi -\int_{\R}V^2U'd\xi-2(k-1)\int_{\R}UVU'd\xi.
 \end{equation}
A quick integration by parts while using the boundary conditions shows that the first term in \eqref{e07} is equal to $-\frac{1}{3}$, the second term, thanks to \eqref{e02} is equal to $-\left(-\frac{1}{3k}\right)$. Using \eqref{e01}, we deduce from \eqref{e07} that 
\begin{equation}\label{e08}
	\begin{array}{lllll}
	\dis 2\int_{\R}G(U,V)U'd\xi +\int_{\R}W^2U'd\xi=\dis -\frac{1}{3}-\left(-\frac{1}{3k}\right)-2(k-1)\left(-\frac{1}{6k}\right)=0.
	\end{array}
\end{equation}
Consequently, $\dis \int_{\R}G(U,V)U'd\xi=\frac{1}{2} \int_{\R}W^2(-U')d\xi$, which combined with \eqref{siha0} gives the first equality in \eqref{imp2}.  Next, combining \eqref{e000}-\eqref{e01} gives 
\begin{equation}\label{help0}
	\int_{\R}UW(-U')d\xi =\frac{k-1}{6k}.
\end{equation}
The Cauchy-Schwarz inequality applied to $U \sqrt{-U'}$ and $W \sqrt{-U'}$ lead us to
\begin{equation*}
\begin{array}{llll}
	\dis 	\left(\frac{k-1}{6 k}\right)^2 & =& \dis \left(\int_{\R} U W(-U') d \xi\right)^2 \\
	& \leq& \dis \left(\int_{\R} U^2(-U') d \xi\right)\left(\int_{\R} W^2(-U') d \xi\right) \\
	& =&\dis \frac{1}{3} \int_{\R} W^2(-U') d \xi.
\end{array}
\end{equation*}

Hence the inequality in \eqref{imp2} holds. Finally, we deduce from \eqref{imp1}, that $d=1$. This completes the proof.
\end{proof}
 We also prove that $U+V<1$ and obtain a quantitative lower bound on the maximum of $1-U-V$. 
\begin{corollary}\label{corollary1}
	Let $k>1$ and $d>0$. If a monotone solution of  \eqref{eq1}-\eqref{eq2a}
	has $c=0$, then 
	\begin{equation}\label{eq5}
		U(\xi)+V(\xi)<1\; \text{ for every   } \xi\in \R,
	\end{equation}
		and 
		\begin{equation}\label{eqnew}
			\max _{\xi \in \mathbb{R}}(1-U(\xi)-V(\xi)) \geq \frac{k-1}{3k} .
		\end{equation}
\end{corollary}
\begin{proof}
	Let $W(\xi)=U(\xi)+V(\xi)$. Then $W(\xi) \to 1$ as $\xi \to \pm\infty$. Now if $c=0$, then combining the equations in  \eqref{eq1} gives  
	$$W''(\xi)=-F(U,V)-\frac{1}{d}G(U,V).$$
	But from Theorem \ref{theo1}, $U,V>0$. Thus for all $\xi\in \R$ such that $W(\xi)\geq 1$ (and since $k>1$),  we have  $1-U-kV=1-W+(1-k)V<0$ and $1-V-kU=1-W+(1-k)U<0$. This implies that 
	\begin{equation}\label{eq00}
		W''(\xi)>0 \quad \text{for all } \xi\in \R\text{ such that }  W(\xi)\geq 1.
	\end{equation}
	Now, assume that $W(\xi_0)>1$ for some $\xi_0 \in \R$.  Since $W(\xi)\to 1$ as $|\xi|\to \infty$, then $W$ attains a global maximum at some $\xi_m \in \R$ and $W(\xi_m)>1$. So $W''(\xi_m)\leq 0$. This contradicts \eqref{eq00}. Consequently, $W(\xi)\leq 1$ for all $\xi \in \R$. If $W(\xi_1)=1$ for some $\xi_1\in \R$. Then $\xi_1$ is a local maximum and again $W''(\xi_1)\leq0$, contradicting \eqref{eq00}. Therefore \eqref{eq5} holds.  Finally if we set $\dis M=\max _{\xi \in \mathbb{R}}(1-U(\xi)-V(\xi))>0$, then from \eqref{help0}, we can deduce that 
	$$\frac{k-1}{6k}\leq M\int_{\R}U(-U')d\xi =\frac{M}{2}.$$
Hence \eqref{eqnew} holds. This completes the proof.
\end{proof}


\section{Proof of Theorem \ref{maintheo}} \label{mainresult}
In this section, we prove our main theorem. We first provide some results needed in the sequel.
The following result is obtained from \cite{rodrigo2001}, where the authors explicitly computed nine families of travelling waves solutions of \eqref{model} (see also \cite{rodrigo2000}). We use their family of travelling front solution No. 6 on their page 668. We provide details here to avoid any convention mismatch. In the symmetric setting, their result reads as follows.
\begin{proposition}\label{proprodrigo}
Let $k=\frac{11}{6}$, $d=\frac{11}{2}$, $c=-\frac{\sqrt{6}}{12}$ and $\lambda=\frac{1}{\sqrt{6}}$. We define
$$
\phi(\xi)=\frac{1}{1+\mathrm{e}^{\lambda \xi}}, \quad U(\xi)=\phi(\xi)^2, \quad V(\xi)=1-\phi(\xi).
$$
Then  $(U,V)$ solves \eqref{eq1}-\eqref{eq2a} with speed $c$. In particular
$
c_{11/6, 11/2}=-\frac{\sqrt{6}}{12}<0.
$
\end{proposition}

\begin{proof}
 We have $\phi'=-\lambda\phi(1-\phi)$ and $\phi''=\lambda^2\phi(1-\phi)(1-2\phi)$. With straightforward computations, we obtain $U'=-2 \lambda \phi^2(1-\phi)$ and $U''=2 \lambda^2 \phi^2(1-\phi)(2-3 \phi)$. Also $F(U,V)=\phi^2\left((1-k)+k \phi-\phi^2\right)$. Therefore, the first equation in \eqref{eq1} becomes after a simplification by $\phi^2>0$
 \begin{equation*}
 	\begin{array}{llllll}
 		 0&=&2 \lambda^2(1-\phi)(2-3 \phi)-2 c \lambda(1-\phi)+(1-k)+k \phi-\phi^2\\
 		&=& (6 \lambda^2-1)\phi^2+(-10 \lambda^2+2 c \lambda+k)\phi+(4 \lambda^2-2 c \lambda+1-k).
 	\end{array}
 \end{equation*} 
This implies that 
\begin{equation}\label{si0}
\lambda^2=\frac{1}{6}, \quad \lambda c=\frac{5}{6}-\frac{k}{2}.
\end{equation}
Moreover, $V'=\lambda\phi(1-\phi)$, $V''=-\lambda^2\phi(1-\phi)(1-2\phi)$ and $G(U,V)=\phi(1-\phi)(1-k\phi)$.  Thus,  the second equation in \eqref{eq1} becomes after a simplification by $\phi(1-\phi)>0$
\begin{equation*}
	\begin{array}{llllll}
		0&=&-d\lambda^2(1-2\phi)+c\lambda +1-k\phi\\
		&=& (2d\lambda^2-k)\phi+(-d\lambda^2+c\lambda+1).
	\end{array}
\end{equation*} 
So, 
\begin{equation}\label{si1}
d=3k, \quad \lambda c=\frac{k}{2}-1.
\end{equation}
Combining \eqref{si0} and \eqref{si1} gives $k=\frac{11}{6}$, $d= \frac{11}{2}$ and $c=-\frac{\sqrt{6}}{12}$. Moreover,  $\phi(-\infty)=1$, $\phi(+\infty)=0$ and $\phi'<0$. Therefore, $U'=2\phi'\phi<0$,  $V'=-\phi'>0$ and the boundary conditions in \eqref{eq2} are satisfied. Hence $(U,V)$ is solution to \eqref{eq1}-\eqref{eq2a} and By Theorem \ref{theo1}, the speed of the monotone front is unique. This implies that $c_{11 / 6,11 / 2}=c$.
\end{proof}

We prove the following result. A similar result is obtained in \cite[Section 6]{ma2019}.

\begin{proposition}\label{proprecipro}
For every $k>1$ and $d>0$ we have
\begin{equation}\label{naam}
		c_{k,d}=-\sqrt{d}\,c_{k,1 / d}.
\end{equation}
\end{proposition}
\begin{proof}
Let us consider a solution $(U,V,c)$ of \eqref{eq1}-\eqref{eq2a} with $d>0$. We make the change of variables  $\eta=-\sqrt{d} \xi$ and we define 
$$
\widetilde{U}(\xi)=V(\eta), \quad \widetilde{V}(\xi)=U(\eta), \quad \widetilde{c}=-\frac{c}{\sqrt{d}} .
$$
As $\xi$ goes from $-\infty$ to $+\infty$, $\eta$ goes from $+\infty$ to $-\infty$. Hence
$$
(\tilde{U}, \tilde{V})(-\infty)=(1,0), \quad(\tilde{U}, \tilde{V})(+\infty)=(0,1) .
$$
Moreover, a differentiation gives
$$
\widetilde{U}'=-\sqrt{d} V^{\prime}(\eta)<0, \quad \widetilde{V}'=-\sqrt{d} U'(\eta)>0,
$$
and
$$
\tilde{U}''=d V''(\eta), \quad \tilde{V}''=d U''(\eta).
$$
The reaction terms give
$$
F(\widetilde{U}, \widetilde{V})=G(U, V)(\eta), \quad G(\widetilde{U}, \widetilde{V})=F(U, V)(\eta) .
$$
Therefore,
\begin{equation*}
	\begin{array}{lll}
	\widetilde{U}''+\widetilde{c} \widetilde{U}'+F(\widetilde{U}, \widetilde{V}) &=d V^{\prime \prime}(\eta)+c V^{\prime}(\eta)+G(U, V)(\eta)=0, \\
	\frac{1}{d} \widetilde{V}''+\widetilde{c} \widetilde{V}'+G(\widetilde{U}, \widetilde{V})  &=U^{\prime \prime}(\eta)+c U^{\prime}(\eta)+F(U, V)(\eta)=0 .
	\end{array}
\end{equation*}
Consequently,  $(\widetilde{U}, \widetilde{V})$ is a monotone front for diffusion ratio $\frac{1}{d}$ and speed $-\frac{c}{\sqrt{d}}$. Uniqueness of the speed given by  Theorem \ref{theo1} gives $c_{k, 1 / d}=-\frac{c_{k, d}}{\sqrt{d}}$.  Hence \eqref{naam} holds.
\end{proof}

We can now prove the Theorem \ref{maintheo}.
\begin{proof}[Proof of Theorem \ref{maintheo}]
We consider the set $I=(1, \infty) \times(1, \infty)$. Then $I$ is convex and therefore connected. Using Theorem \ref{theosmooth}, the map $(k, d) \mapsto c_{k, d}$ is continuous on $I$. By Theorem \ref{zerospeed}, it does not vanish on $I$. The Lemma \ref{lem1} shows that its sign is constant on $I$. The point $(11/6, 11/2)$ belongs to $I$, and thanks to Proposition \ref{proprodrigo}, we have $c_{11 / 6,11 / 2}=-\frac{\sqrt{6}}{12}<0$. Hence $c_{k, d}<0$ for every $k>1$ and $d>1$. 

Furthermore, if  $d=1$,  then Proposition \ref{proprecipro} gives $c_{k,1}=-c_{k,1}$. That is $2 c_{k,1}=0$, and so $c_{k,1}=0$.
Finally, for $0<d<1$, we have $1/d>1$ and the arguments above give $c_{k,1 / d}<0$. From \eqref{naam}, we deduce that $	c_{k,d}=-\sqrt{d}\,c_{k,1 / d}>0.$ This completes the proof.
\end{proof}

\begin{remark} The Theorem \ref{maintheo} proves that the faster diffuser advances. This implies that the front speed favours the species with the larger diffusion coefficient when all reaction parameters are identical and $k>1$. We mention that the theorem does not imply that the faster diffuser wins for every initial condition. Indeed, the reaction system is bistable and any spatially homogeneous initial data in the basin of either exclusion equilibrium would converge to that equilibrium independently of the diffusion coefficient. 	
We also mention that it does not mean that natural selection always favours high dispersal. Our symmetric model \eqref{model} treats dispersal as fixed and therefore excludes the evolution of dispersal, reproduction, and competitive ability, together with trade-offs among these traits at expanding fronts (see e.g., \cite{burton2010}). We refer to \cite{girardin2017, girardin2019, girardin2018, he2013} for discussions on the distinction between homogeneous and heterogeneous competition.
\end{remark}

%


\subsection*{Acknowledgements}
CK acknowledges support from the Natural Sciences and Engineering Research Council of Canada (NSERC), Discovery Grant RGPIN-2025-05864 and the Discovery Launch Supplement DGECR-2025-00184. 


\bibliographystyle{abbrv}
\bibliography{References}

\appendix 
\section{A positivity lemma for a linear equation on a half-line}
In this appendix we prove the elementary result used in the proof of  Proposition \ref{propposi}. It plays the role of a maximum principle on the unbounded interval $[a,+\infty)$, the compactness of the domain being replaced by the decay of $w$ at infinity. Since the classical statements are formulated on bounded domains, we include a short proof. 
\begin{lemma}\label{lempos}
	Let $D>0$, $c\in\R$, $a\in\R$ and let $q\in C([a,+\infty))$ with $q\leq0$ on $[a,+\infty)$. If
	$w\in C^2([a,+\infty))$ satisfy
	\begin{equation}\label{eqw}
		Dw''+cw'+qw=0 \text{ on }[a,+\infty),\qquad w(a)>0,\qquad w(+\infty)=0,
	\end{equation}
then $w>0$ on $[a,+\infty)$.
\end{lemma}

\begin{proof}
Let $Lw:=D w''+cw'+qw$. Suppose first that $w(\xi_0)<0$ for some $\xi_0>a$. Since $w(\xi)\to 0$ as $\xi\to \infty$, there exists $R>\xi_0$ so that $w(R)>w(\xi_0)$. The function $w$ being continuous on the compact interval $[a,R]$, it attains its minimum there. Since $w(a)>0>w(\xi_0)$ and $w(R)>w(\xi_0)$, the minimum is attained at some interior point $\xi_1\in(a,R)$, and $w(\xi_1)\leq w(\xi_0)<0$. Thus $-w$ attains at $\xi_1$ a positive interior maximum. Moreover, $L(-w)=0$ and $q\leq0$. Thus  the strong maximum principle implies that $-w$ is constant on $(a,R)$, hence on $[a,R]$ by continuity. This contradicts $w(a)>0>w(\xi_1)$. Hence $w\geq0$ on $[a,+\infty)$.
	
Now, if we assume that $w(\xi_2)=0$ for some $\xi_2>a$. Then $\xi_2$ is an interior minimum of $w$, so that $w'(\xi_2)=0$. Since  \eqref{eqw} is linear with continuous coefficients, the uniqueness for the Cauchy problem gives $w\equiv0$, which contradicts $w(a)>0$. Hence $w>0$ on $[a,+\infty)$.
\end{proof}
\end{document}